\documentclass[12pt]{article}

\usepackage{amsmath}
\usepackage{amssymb}
\usepackage{amsthm}
\usepackage{amsfonts}
\usepackage{latexsym}
\usepackage{cite}
\usepackage{mathrsfs}
\usepackage{url}
\usepackage{color}
\usepackage{eufrak}
\usepackage{hyperref}

\usepackage{scalerel}
\usepackage{tikz}
\usetikzlibrary{svg.path}

\definecolor{orcidlogocol}{HTML}{A6CE39}
\tikzset{
  orcidlogo/.pic={
    \fill[orcidlogocol] svg{M256,128c0,70.7-57.3,128-128,128C57.3,256,0,198.7,0,128C0,57.3,57.3,0,128,0C198.7,0,256,57.3,256,128z};
    \fill[white] svg{M86.3,186.2H70.9V79.1h15.4v48.4V186.2z}
                 svg{M108.9,79.1h41.6c39.6,0,57,28.3,57,53.6c0,27.5-21.5,53.6-56.8,53.6h-41.8V79.1z M124.3,172.4h24.5c34.9,0,42.9-26.5,42.9-39.7c0-21.5-13.7-39.7-43.7-39.7h-23.7V172.4z}
                 svg{M88.7,56.8c0,5.5-4.5,10.1-10.1,10.1c-5.6,0-10.1-4.6-10.1-10.1c0-5.6,4.5-10.1,10.1-10.1C84.2,46.7,88.7,51.3,88.7,56.8z};
  }
}

\newcommand\orcidicon[1]{\href{https://orcid.org/#1}{\mbox{\scalerel*{
\begin{tikzpicture}[yscale=-1,transform shape]
\pic{orcidlogo};
\end{tikzpicture}
}{|}}}}

\usepackage{hyperref} 

\makeatletter
\@addtoreset{equation}{section}
\makeatother

\newtheorem{theorem}{Theorem}[section]
\newtheorem{lemma}[theorem]{Lemma}

\newtheorem{proposition}[theorem]{Proposition}
\theoremstyle{definition}
\newtheorem{definition}[theorem]{Definition}
\newtheorem{remark}[theorem]{Remark}
\newtheorem{example}[theorem]{Example}

\newcommand{\A}{\mathcal{A}}
\newcommand{\V}{\mathcal{V}}
\newcommand{\Wc}{\mathcal{W}}

\newcommand{\F}{\mathbb{F}}
\newcommand{\Sk}{\mathfrak{S}}

\newcommand{\As}{\mathscr{A}}
\newcommand{\Bs}{\mathscr{B}}
\newcommand{\Ds}{\mathscr{D}}
\newcommand{\G}{\mathscr{G}}

\newcommand{\ov}{\overline{v}}
\newcommand{\oc}{\overline{c}}
\newcommand{\ou}{\overline{u}}
\newcommand{\ozero}{\overline{0}}

\newcommand{\T}{\text}
\newcommand{\db}{\displaybreak[3]}

\begin{document}
\title{
New infinite families of uniformly packed near-MDS codes and multiple coverings, based on the ternary Golay code
\date{}
\thanks{The research of S. Marcugini and F. Pambianco was supported in part by the Italian
National Group for Algebraic and Geometric Structures and their Applications (GNSAGA -
INDAM) (Contract No. U-UFMBAZ-2019-000160, 11.02.2019) and by University of Perugia
(Project No. 98751: Strutture Geometriche, Combinatoria e loro Applicazioni, Base Research
Fund 2017--2019; Fighting Cybercrime with OSINT, Research Fund 2021).}
}
\maketitle
\begin{center}
{\sc Alexander A. Davydov \orcidicon{0000-0002-5827-4560}}\\
 {\sc\small Kharkevich Institute for Information Transmission Problems}\\
 {\sc\small Russian Academy of Sciences,
Moscow, 127051, Russian Federation}\\
 \emph{E-mail address:} alexander.davydov121@gmail.com\medskip\\
 {\sc Stefano Marcugini \orcidicon{0000-0002-7961-0260}and
 Fernanda Pambianco \orcidicon{0000-0001-5476-5365}}\\
 {\sc\small Department of  Mathematics  and Computer Science,  Perugia University,}\\
 {\sc\small Perugia, 06123, Italy}\\
 \emph{E-mail address:} \{stefano.marcugini, fernanda.pambianco\}@unipg.it
\end{center}

\textbf{Abstract.}
We present five new infinite families of linear near-MDS codes uniformly packed in the wide sense (UPWS). These codes are also almost perfect multiple coverings of the deep holes or farthest-off points (APMCF), i.e.\ the vectors lying at distance $R$ (covering radius) from the code. The families are constructed by $m$-lifting when one takes a starting code $C$ over the ground Galois field $\F_q$ with a parity check matrix $H(C)$ and then considers the codes $C_m$ over $F_{q^m}$, $m\ge2$, with the same parity check matrix $H(C)$. As starting codes we used the ternary perfect Golay code and codes obtained by its extension and puncturing. To prove the needed combinatorial properties (UPWS and APMCF), we used the $m$-lifting  of the dual codes and features of near-MDS codes. A general theorem on infinite families of UPWS near-MDS codes is proved.



\section{Introduction and background}\label{sec1:Intro}
Let $\F_q$ be the Galois field with $q$ elements, $\F_q^*=\F_q\setminus\{0\}$. Let $\#S$ be the cardinality of a set $S$.  Let a $t$-set  be a set  of cardinality~$t$. Let $\triangleq$ be the sign ``equality by definition''. We use lower-case letters, e.g. $v$, for scalars, and overlined lower-case letters, e.g. $\ov$, for vectors. Let $\ozero$ be the zero vector. The $n$-dimensional vector space over $\F_{q}$ is denoted by $\F_{q}^{\,n}$. Let $\F_q^{\,m\times n}$ be the set of $m \times n$ matrices with entries from $\F_q$.

 The \emph{Hamming distance} $d(\ov,\ov')$ between two vectors $\ov$, $\ov'$ is the number of coordinates in which they differ. The \emph{Hamming ball} of radius $\rho$ with center $\ov\in \F_{q}^{\,n}$ is the set $\Bs_q(\rho,n,\ov)\triangleq\{\ov'\,|\,\ov'\in \F_{q}^{\,n},~d(\ov,\ov')\leq \rho\}\subseteq\F_{q}^{\,n}$.
The \emph{Hamming weight} $\T{wt}(\ov)$ of the vector $\ov\in\F_{q}^{\,n}$ is the number of its non-zero coordinates.

A \emph{linear code} $C$ over $\F_q$ of length $n$ and dimension $k$ is a $k$-dimensional subspace of $\F_{q}^{\,n}$.  A vector $\oc\in C$ is a \emph{codeword} of $C$.  The \emph{minimum distance} (or simply \emph{distance}) $d$ of $C$ is the minimum Hamming distance between any pair of codewords. The \emph{covering radius} $R$ of $C$ is the smallest integer such that the space $\F_{q}^{\,n}$ is covered by the Hamming balls of radius $R$ centered at codewords.
Such a code $C$ is denoted by $[n,k,d]_qR$, where $d$ and $R$ can be omitted. The \emph{packing radius} of $C$ is $e(C)\triangleq\left\lfloor(d-1)/2\right\rfloor$. If $R=e(C)$, then $C$ is a \emph{perfect code}; if $R=e(C)+1$, then $C$ is a \emph{quasi-perfect code}.

The number of codewords of the code $C$ with Hamming weight $i$ is denoted by $A_i(C)$. The set $\{A_i(C)\,|\,0\le i\le n\}$ is the \emph{weight distribution} of~$C$.
The support of the set of the non-zero weights $A_i(C)$ of the code $C$ is denoted by $S(C)$, \emph{the number of the non-zero weights} in the code $C$ is denoted by $s(C)$, i.e.
\begin{equation}\label{eq1:support non-zero weights}
 S(C)\triangleq\{i>0\,|\,A_i(C)\ne0\}, ~s(C)\triangleq\#S(C).
\end{equation}

Let $C$ be an $[n, k,d]_qR$ code with codewords $\oc$. The $[n, n-k,d^\bot]_qR^\bot$ code $C^\bot$ \emph{dual} to $C$ is the set
$C^\bot\triangleq\{\ov\in\F_q^{\,n}\,|\, \langle\ov\cdot\oc\rangle=\ozero\in\F_q^{\,n}\T{ for every }\oc\in C\}$
where $\langle\ov\cdot\oc\rangle$ denotes the inner product of vectors $\ov$ and $\oc$ in $\F_q^{\,n}$;  a code $C$ is \emph{self-dual} if $C=C^\bot$. The value $s(C^\bot)$ is the \emph{external distance} of the code $C$ \cite{Delsarte}.

For a vector $\ov$ of $\F_q^{\,n}$ and a code $C$, let $d(\ov,C)\triangleq\min\limits_{\oc\in C}d(\ov,\oc)$ be the Hamming \emph{distance between $\ov$ and $C$} and let $B_j(\ov,C)$ be \emph{the number of codewords of $C$ at distance $j$ from $\ov$}.

A code $C$ \emph{punctured on position $i,j,k$} is denoted $C_{i,j,k}^*$; it can be obtained by deleting the same positions $i,j,k$ in each codeword of $C$.
The \emph{extended} $[n, k,d]_qR$ code $C$ is the $[n+1,k,\widehat{d}\,]_q\widehat{R}$ code $\widehat{C}$ obtained by adding an overall parity check digit in each codeword.

For an introduction to coding theory, see \cite{HufPless,MWS,Roth}; about covering radius, see
\cite{CHLL-bookCovCod,Handbook-coverings,Delsarte,DGMP-AMC,HufPless}; see also Section \ref{subsec21:CosetCRlift}.

\begin{definition}\label{def1:lift} \cite[Definition 4]{GCR2022IEEE-EWS}, \cite[Definition 11]{FaldWillSmallDef1997}, \cite{RiZinLiftPerf}
  Let $C$ be an $[n,k,d]_qR$ code over the ground field $\F_q$ with a parity check matrix $H(C)\subset\F_q^{\,(n-k)\times n}$.  The \emph{$m$-lifted code} $C$ is the  $[n,k]_{q^m}R_m$ code, say $C_m$, over the field $\F_{q^m}$ with covering radius $R_m$ and with \emph{the same parity check matrix $H(C)$}, $m\ge1$. We put $C_1=C$, $R_1=R$. It is possible that $R_m\ne R$ if $m\ge2$. We say that the code $C_m$ is obtained by \emph{$m$-lifting} $C$. In the literature, see e.g. \cite[Definition 11]{FaldWillSmallDef1997},  \cite {JurPelCoset2015}, the term``extension'' is used instead of ``lifting''.
\end{definition}

\begin{proposition}\label{prop1:m-lifted} \emph{\cite[Lemma 5]{GCR2022IEEE-EWS}, \cite[Proposition 12]{FaldWillSmallDef1997}, \cite[Lemma 2.2]{RiZinLiftPerf}}
 Let $C_m$ be the $[n,k]_{q^m}R_m$ code, obtained by $m$-lifting an $[n,k,d]_qR$ code $C$. Then
 for any $m\ge 1$,  the minimum distance of $C_m$ is the same as in $C$, i.e. $C_m$ is an $[n,k,d]_{q^m}R_m$ code.
\end{proposition}

There are three different concepts of uniformly packed codes. The uniformly packed
codes in the narrow sense were introduced in 1971 in \cite{SemZinovZaitsUP}. They are a subclass of uniformly packed codes introduced in 1975 in \cite{GoethTilb1975UP}. The class of uniformly packed codes in the wide sense (UPWS, for short) was introduced in 1974 in \cite{BasZaiZinUP}.
Such codes are interesting as they have good packing density and possess combinatorial properties
similar to perfect codes \cite[Introduction]{BorgRiZin2019PIT}. In this paper we are interested in the UPWS codes.

\begin{definition}\label{def1:UnPackWide} \cite{BasZaiZinUP}, \cite[Definition 6]{BorgRiZin2019PIT}, \cite[Definition 1.2]{RiZinLiftPerf}
  An $[n,k,d]_qR$ code $C$
  is \emph{uniformly packed in the wide sense} (UPWS) if there exist rational numbers $\beta_0,\beta_1,\ldots,\beta_R$ (packing parameters) such that for any vector $\ov\in\F_{q}^{\,n}$ we have $\sum_{j=0}^R\beta_j B_j(\ov,C)=1$.
\end{definition}

\begin{theorem} \label{th1:UPWS nes-suf} \emph{\cite[Theorem 1]{BasZin}, \cite[Lemma 2.2(ii)]{BorgRiZin12010}, \cite[Theorem 4(iv)]{BorgRiZin2019PIT}, \cite[Lemma 3.1]{RiZinLiftPerf}}
A code $C$ of covering radius $R(C)$ is UPWS if and only if $ R(C)=s(C^\bot)$.
\end{theorem}

Codes UPWS are essentially connected with completely regular (CR, for short) codes, see for example
\cite{BasZaiZinUP,BasZin,BorgRiZin12010,BorgRiZin2019PIT,CameronvanLint,%
Delsarte,GoethTilb1975UP,NeumaierCR,%
RiZinLiftPerf,SemZinovZaitsUP,TilborgThes} and Section~\ref{subsec21:CosetCRlift}.

\begin{definition}\label{def1:multcov} \cite{BDGMP_MultCov,BDGMP_MultCovFurth,CHLL-bookCovCod,DMP-AMC-Cosets} The vectors  $\ov\in\F_q^{\,n}$ with $ d(\ov,C)=R$ are called \emph{the farthest-off points} or \emph{deep holes}.
 An $[n,k,d]_{q}R$ code $C$ is said to be an \emph{$(R,\lambda)$
multiple covering} of the farthest-off points ($(R,\lambda )$-MCF
code, for short) if for all vectors $\ov$ with $d(\ov,C)=R$ we have $B_R(\ov,C)\ge\lambda$. If for all vectors $\ov$ with $d(\ov,C)=R$
we have $B_R(\ov,C)=\lambda$ then the code $C$ is called an \emph{$(R,\lambda)$
almost perfect multiple covering} of the farthest-off points
($(R,\lambda )$-APMCF code, for short); if in addition $d\geq2R$ then the code is called
an \emph{$(R,\lambda )$  perfect multiple covering} of the farthest-off
points $((R,\lambda )$-PMCF code, for short).
The parameter $\lambda$ is called \emph{the multiplicity of covering.}
 \end{definition}

Motivations for studying MCF codes arise e.g.\ from the generalized football pool problem \cite{HamalHonkLitsOster,HamalRankGolay} and the list decoding \cite[Chapter 9]{Roth}, \cite{WeeCoLit}.

The perfect ternary Golay code $\G_{11}=[11,6,5]_32$ \cite{Golay1949} has an interesting history, see \cite[p.\ 25]{Barg}, \cite[Introduction, p.\ 85]{HamalRankGolay}, \cite{ThompsonGolay}, and references therein.

 \emph{In this work}, we consider a few ``starting'' codes, connected with the ternary Golay code $\G_{11}$, and then we examine the corresponding $m$-lifted codes for $m\ge2$ which generate infinite code families. Our investigations concentrate on combinatorial properties of the codes, first of all, whether they are UPWS  and/or $(R,\lambda)$ APMCF. As a result, we proved for five infinite families that codes in them are UPWS and also are $(R,\lambda)$ APMCF. As far as it is known to the authors, these properties of the code families have not been described in the literature. We also obtained several new sporadic results; for example a few codes were shown to be $(R,\lambda)$ PMCF. For the new results, see Theorems \ref{th41:UPWSliftGolay}, \ref{th42:UPWSliftGolay} and Section~\ref{sec6:conclusion}.

 We use the following codes as the starting ones: the ternary perfect Golay code $\G_{11}=[11,6,5]_32$; the self-dual quasi-perfect extended ternary Golay code $\G_{12}=\widehat{\G}_{11}=[12,6,6]_33$; the codes
$\G_{10}=(\G_{11})^*_{11}=[10, 6, 4]_32$, $\G_{9}=(\G_{11})^*_{11,10} =[9, 6, 3]_32$,  $\G_{8}=(\G_{11})^*_{11,10,9} =[8, 6, 2]_31$, obtained by puncturing $\G_{11}$;
and, finally, the dual codes $\G_{11}^\bot=[11,5,6]_35$, $\G_{10}^\bot=[10,4,6]_35,~\G_{9}^\bot=[9,3,6]_35,~\G_{8}^\bot=[8,2,6]_35$. Then, for $m\ge2$ we investigate the corresponding $m$-lifted codes $\G_{n,m}$, $\G_{n,m}^\bot$, $n=12,11,10,9,8$, and prove that
for the infinite families $\G_{n,m}$ the condition $ R(C)=s(C^\bot)$ of Theorem \ref{th1:UPWS nes-suf} holds if $m\ge m'(n)$ where $n-6\ge m'(n)\ge 2$. This implies that the codes of these families are UPWS. Also, we prove that an UPWS code of covering radius $R$ always is an $(R,\lambda)$ APMCF, see Theorem~\ref{th52:RUWS=APMCF}. Moreover, the codes $\G_{12},\G_{11},\G_{10},\G_{8}$ are $(R,\lambda)$ PMCF.

To check whether the equality $ R(C)=s(C^\bot)$ holds, we interpret the codes under consideration as near-MDS (NMDS) ones,
 see \cite{BoerAMDS,DingTangDesignsNMDS2020IEEE,DodunLandgNMDS1995,DodunLandgNMDS2000,FaldWillSmallDef1997,LandRousNMDS2015}. This approach turned out to be very useful. We obtained several new results, e.g. the formula of the weight distribution of an $m$-lifted NMDS code \eqref{eq3:weight distrib NMDS}, Theorem \ref{th3:NMDS-Adp1=0-UPWS} (general theorem on infinite families of UPWS NMDS codes), and so on.  The properties of NMDS codes are used in the proofs of almost all lemmas, propositions, and theorems in Sections \ref{sec3:useful} and~\ref{sec4:newUPWS}.

Many of our starting codes are well studied, but this is not the case for their $m$-lifted versions. Nevertheless, we obtained new results not only for $m$-lifted codes but also for starting ones, e.g. assertions connected with multiple coverings and the number of codewords of minimal weight $d$. The new results are summarized in
Theorems \ref{th41:UPWSliftGolay}, \ref{th42:UPWSliftGolay}, and Section \ref{sec6:conclusion}.

The paper is organized as follows. Section \ref{sec2:prelim} contains preliminaries. Sections \ref{sec3:useful} and~\ref{sec4:newUPWS} are devoted to infinite families of UPWS codes. Section \ref{sec5:MultCover} considers multiply coverings of
the farthest-off points. In Section \ref{sec6:conclusion}, conclusion and open problems are given.

\section{Preliminaries}\label{sec2:prelim}

\subsection{Code cosets, completely regular and lifted codes}\label{subsec21:CosetCRlift}

A \emph{coset} $\V(C)$ of a code $C\subset\F_q^{\,n}$ is defined as
\begin{equation}\label{eq21:coset}
\V(C)=\{\ov\in\F_q^{\,n}\,|\,\ov=\oc+\ou,\oc\in C\}=C+\ou\subset\F_q^{\,n},
\end{equation}
where $\ou\in \V(C)$ is a vector fixed for the given representation; $\#\V(C)=q^{k}$.

 The \emph{weight $\Wc(\V)$ of a coset} $\V$ is the smallest Hamming weight of any vector in~$\V$. Let $\V^{(\Wc)}$ or $\V^{(\Wc)}(C)$ be a coset of weight $\Wc$. Let $\A_w(\V)$ be the number of vectors of Hamming weight $w$ in a coset $\V$. The set $\{\A_w(\V)\,|\,0\le w\le n\}$ is the \emph{weight distribution} of a coset $\V$.
   A \emph{coset leader} is a vector in the coset having the smallest Hamming weight.
   In a coset $\V^{(\Wc)}$ the number of all the coset leaders is $\A_\Wc(\V^{(\Wc)})$.
 The code $C$  is the coset of weight zero. All $q^{n-k}$ cosets of $C$ form a partition of $\F_q^{\,n}$.

 \begin{theorem}\label{th22:cosets}
 We consider an $[n,k,d]_qR$ code $C$ and its cosets $\V(C)$. We have
\begin{description}
    \item[(i)] $\lfloor \frac{d-1}{2}\rfloor\le R\le n-k$ \emph{\cite{Delsarte},\cite[Chapter 11]{HufPless}}; \emph{\textbf{(Delsarte Bound)}} $R\le s(C^\bot)$  \emph{\cite{Delsarte}}.

  \item[(ii)]
  Each coset of $C$ of weight $s(C^\bot)$ has the same weight distribution \emph{\cite[Theorem~7.5.2]{HufPless}}.

    \item[(iii)] For all vectors $\ov\in\V(C)$, we have $B_{j}(\ov,C)=\A_j(\V(C))$, $0\le j\le n$.
 If $\ov\in\V(C)^{(\Wc)}$ then $d(\ov,C)=\Wc$.
\end{description}
 \end{theorem}

Let $C(i)$ be the set of all vectors of $\F_q^{\,n}$ at distance $i$ from a code $C$ \cite{BasZaiZinUP,BorgRiZin2019PIT,NeumaierCR}. We have $C(0)=C$. Also, $C(i)$ can be viewed as the union of all weight $i$ cosets, i.e.
\begin{equation}\label{eq22:C(i)}
  C(i)\triangleq\{\ov\in\F_q^{\,n}\,|\,d(\ov,C)=i\}=\bigcup_{\Wc=i}\V^{(\Wc)}(C),~i=0,1,2,\ldots,R.
\end{equation}

  \begin{definition} \label{def22:t regular CR}
  \begin{description}
    \item[(i)] \cite{GoethTilb1975UP} A code $C$ of covering radius $R$ is $t$-regular ($0 \le t \le R$) if for all $\ov$  with $d(\ov,C) \le t$ and all $j = 0, \ldots, R$,
the value $B_j(\ov,C)$ depends only on $j$ and on $d(\ov,C)$. In other words, a code $C$ of covering radius $R$ is \emph{$t$-regular} if for $0\le\Wc\le t$, all the cosets of $C$ of weight $\Wc$ have the same weight distribution, i.e. all the cosets included in a set $C(i)$, $0\le i\le t$, have the same weight distribution.
    \item[(ii)]  \cite{Delsarte} A code $C$  of covering radius $R$ is \emph{completely regular} (CR), if
it is $R$-regular.
  \end{description}
  \end{definition}

\begin{proposition}\label{prop23:examplesCR}
\begin{description}
  \item[(i)] \emph{\cite[Lemma 2.2]{BorgRiZin12010},\cite[Corollary 2]{BorgRiZin2019PIT}} If a code is CR, then it is UPWS.

  \item[(ii)] \emph{\cite[Section 5.12, p.\ 36]{BorgRiZin2019PIT},\cite{Delsarte}} The dual code of any linear two-weight code is a CR (and hence, UPWS) code with covering radius $R=2$ .

  \item[(iii)] \emph{\cite[Section 5.1]{BorgRiZin2019PIT}} The Golay code $\G_{11}=[11,6,5]_32$, the punctured Golay code $\G_{10}=[10,6,4]_32$, the extended Golay code $\G_{12}=[12,6,6]_33$ are CR (and hence, UPWS).

  \item[(iv)] \emph{\cite[Section 5.10]{BorgRiZin2019PIT}}  Let $H_m^q$ be a parity check matrix (up to monomial equivalence) of the $[n_m,n_m-m,3]_q1$ Hamming code with $n_m=(q^m-1)/(q-1)$. Let $C$ be an $[\ell n_m,\ell n_m-m,d(C)]_qR(C)$ code with a parity check matrix $H=\underbrace{H_m^q\ldots H_m^q}_{\ell\T{ times}}$, $\ell\ge2$. Then $d(C)=2$, $R(C)=1$, and $C$  is a CR (and hence, UPWS) code.
  \end{description}
\end{proposition}

\begin{proposition}\label{prop24:m-lifted} \emph{\cite{BorgRiZin12010,RiZinLiftPerf}}
 Let $C_m$ be the $[n,k,d]_{q^m}R_m$ code over the field
$\F_{q^m}$, obtained by $m$-lifting an $[n,k,d]_qR$ code $C$ over $\F_q$, see Definition \emph{\ref{def1:lift}}. Then
$  C_m=C+\alpha C+\alpha^2 C+\ldots+\alpha^{m-1} C,$
where $\alpha$ is a primitive element of $\F_{q^m}$ \emph{\cite[Proposition 2.1]{RiZinLiftPerf}}. Also, for any
codeword $\oc\in C_m$ of weight $d$, there exists a codeword
$\oc'\in C$ of weight  $d$ such that $\oc=\beta\oc'$, where $\beta\in\F_{q^m}^*$ \emph{\cite[Lemma 2.2]{RiZinLiftPerf}}.
Finally, the code $C_m$ is self-dual $(C_m=C_m^\bot)$ if and only if $C$  is self-dual \emph{\cite[Lemma 3.6]{BorgRiZin12010}}.
\end{proposition}

\subsection{Near-MDS (NMDS) codes}\label{subsec22:nearMDS}
NMDS codes have been introduced in 1995 in \cite{DodunLandgNMDS1995}, see also \cite{BoerAMDS,DingTangDesignsNMDS2020IEEE,DodunLandgNMDS2000,FaldWillSmallDef1997,LandRousNMDS2015}, \cite[p. 363]{Roth}, and   references therein. In \cite{FaldWillSmallDef1997}, NMDS codes are called "dually AMDS Codes".

\begin{definition} \label{def25:nearMDS} \cite[Lemma 3.1, Corollary 3.3]{DodunLandgNMDS1995}, \cite[Definitions 2.2, 2.4]{DodunLandgNMDS2000}
 A linear $[n,k,n-k]_q$  code $C$ is an NMDS code if and only if  its dual $C^\bot$ is an $[n,n-k,k]_q$ code. In other words, a linear $[n,k,d]_q$  code $C$ is an NMDS code if and only if $d(C)+d(C^\bot)=n$.
\end{definition}

\begin{theorem}\label{th25:NMDS}
 \begin{description}
\item[(i)] \emph{\cite[Lemma 3.2]{DodunLandgNMDS1995}} A code dual to an NMDS one is NMDS as well.

 \item[(ii)] \emph{\cite[Theorem 3.4]{DodunLandgNMDS1995}, \cite[Theorem 2.5]{DodunLandgNMDS2000}}
 If $n>k + q$, every $[n,k,n-k]_q$ code is NMDS.

\item[(iii)] \emph{\cite[Proposition 5.3, Proof]{DodunLandgNMDS1995}} Let $C$ be an $[n,k,n-k]_q$ NMDS code with  a generator matrix $G(C)$. Then removing a column from $G(C)$
    with preserving a set of $k+1$ columns which contain $k$ linearly dependent columns
    gives a generator matrix of an $[n-1,k,n-k-1]_q$ NMDS code.

\item[(iv)]  \emph{\cite[Proposition 14]{FaldWillSmallDef1997}}
The number of codewords of minimum weight in an NMDS code $C$ and its dual $C^\bot$ are equal.

   \item[(v)]
Let $C$ be an $[n,k,d]_q=[n,k,n-k]_q$ NMDS code  with $d=n-k$, $A_d(C)=A_{n-k}(C)$. For the codeword weights of $C$ we have the following.
\begin{align}\label{eq25:weight distrib NMDS}
& \bullet~A_{d+s}(C)=\binom{n}{k-s}\sum_{j=0}^{s-1}(-1)^j\binom{d+s}{j}(q^{s-j}-1)+(-1)^s\binom{k}{s}A_{d}(C),\db\\
&\phantom{\bullet~~}s\ge1 \T{\emph{\cite[Theorem 4.1]{DodunLandgNMDS1995}, \cite[Corollary 15]{FaldWillSmallDef1997}}}.\notag\db\\
&\bullet~A_d(C)\le A^{\max}_{d}(n,k,q)\triangleq\binom{n}{k-1}\frac{q-1}{k}\T{ with  equality} \label{eq25:UpBndAd NMDS}\db\\
&\phantom{\bullet~~}A_{d}(C)=A^{\max}_{d}(n,k,q) \T{ if and only if }A_{d+1}(C) = 0\T{ \emph{\cite[Corollary 4.2]{DodunLandgNMDS1995}, \cite{LandRousNMDS2015}}}.\notag\db\\
&\bullet~A_{d}(C)=\binom{n}{d-2}(q^2-1)\binom{d}{2}^{-1} \T{ if }  n=2q+d,~d\ge q+1;\label{eq25:AdNMDS}\db\\
&\phantom{\bullet~~}A_{d}(C)=\binom{n}{d-2}q(q-1)\binom{d}{2}^{-1} \T{ if }  n=2q+d-1,~d\ge q+1 \emph{\cite[Lemma 2]{BoerAMDS}}.\notag\db\\
&\bullet~0\le A_{d+2}(C)\le\frac{q-1}{2}\binom{n}{k-2}(2q+k-n) \emph{\cite[Proposition 5.1, Proof]{DodunLandgNMDS1995}}.\label{eq25:Adp2 NMDS}\db\\
&\bullet~\T{If $n=2k$, the weight distributions of $C$ and $C^\bot$ coincide} \emph{\cite[Remark, p.\ 37]{DodunLandgNMDS1995}}.\notag
\end{align}

 \item[(vi)] \emph{\cite[Proposition 5.1]{DodunLandgNMDS1995}, \cite[Theorem 2.7, pp.\ 58,59]{DodunLandgNMDS2000}} Let $m'(k,q)$ be the maximum possible length of an NMDS code of dimension $k$ over the field $\F_q$. We have
     $ m'(k,q)\le2q+k$; an $[2q+k,k,2q]_q$ NMDS code with $n=2q+k$ is called
        \emph{extremal}, it has $A_{d+1}=A_{2q+1}=0$;
       $m'(2,q)=2q+2$; $m'(3+t,3)=9+t,~t=0,1,2,3$.
     Moreover, the NMDS codes realizing the values $m'(2,q)$ and $m'(3+t,3)$ are unique up to equivalence; for $m'(2,3)$ this is the $[8,2,6]_3$ code, for $m'(3+t,3)$ this is the $[9,3,6]_3$ code  and its successive extensions $[10,4,6]_3$, $[11,5 ,6]_3$ (the dual to ternary Golay code), $[12,6,6]_3$ (the extended ternary Golay code).

 \item[(vii)]  \emph{\cite[Proposition 23]{FaldWillSmallDef1997}}
Let $C$ be an $[n,k,d]_q$ NMDS code over $\F_q$ with no codewords
of weight $d+1$. If $C_m$ is the $m$-lifted code over $\F_{q^m}$ as in Definition \emph{\ref{def1:lift}}, then $C_m$ is an
$[n,k,d]_{q^m}$ NMDS  code with no codewords of weight $d+1$ as well.
 \end{description}
\end{theorem}

\subsection{Lower bounds on the number of minimum weight codewords}\label{subsec23:minweight}
In \cite{DavTombWordsMinWeght}, the results are given for non-linear codes. In Theorem~\ref{th23:Ad low bnds}, we give these results on linear codes language. In \cite{DavTombWordsMinWeght}, $b_d(n,M,q)$  denotes the number of codewords couples at distance $d$ from each other; we replace it by $A_{d}(n,k,q)$.  Also, in the notations  of the type $\Psi_{d}(n,M,q)$ from \cite{DavTombWordsMinWeght}, we change  $M$ by $k$. In the rest of cases, we write $q^k$ instead of the code cardinality $M$.

Let $ V_{q}(\rho,n)$ be the volume of the Hamming ball $\Bs_q(\rho,n,\ov)$, see Section \ref{sec1:Intro}. The \emph{surface} of this ball is the set $\{\ov'\,|\,\ov'\in F_{q}^{\,n},~d(\ov,\ov')= \rho\}\subset\F_{q}^{\,n}$; its size is denoted by $\Sk_{q}(\rho,n)$. It is well-known that for any $\ov$ we have
\begin{equation*}
  V_{q}(\rho,n)=\sum_{i=0}^\rho\binom{n}{i}(q-1)^i,~
  \Sk_{q}(\rho,n)=\binom{n}{\rho}(q-1)^\rho.
\end{equation*}
\begin{theorem} \label{th23:Ad low bnds}  \emph{\cite{DavTombWordsMinWeght}}
\textbf{\emph{(i)}} Let $C$ be an $[n,k,d]_qR$ code over $\F_q$ of length $n$, cardinality $q^k$, packing radius $e=\lfloor(d-1)/2\rfloor$, minimum distance $d$, covering radius $R$, with $A_d(C)$ codewords of weight $d$.
Let $A_{d}(n,k,q)$ be $\min\limits_C A_d(C)$. The following holds.
\begin{align}
&d=2e+2,~A_{2e+2}(n,k,q)\ge\Psi_{2e+2}(n,k,q)\triangleq\frac{\Sk_{q}(e+1,n)}{\binom{2e+2}{e+1}}
\left(\frac{q^k\Sk_{q}(e+1,n)}{q^{n}-q^kV_{q}(e,n)}-1\right)\label{eq23:A2ep2UP};\db\\
&\T{it is achieved if and only if }C\T{ is a quasi-perfect UPWS code} \emph{\cite[Theorem 2]{DavTombWordsMinWeght}}.\notag\db\\
&d=2e+1,A_{2e+1}(n,k,q)\ge n\Psi_{2(e-1)+2}(n-1,k,q)(2e+1)^{-1}\emph{\cite[Corollaries 1,2]{DavTombWordsMinWeght}}.\label{eq23:A2ep2ReducOdd}\db\\
&d=2e+1,~A_{2e+1}(n,k,q)\ge\Psi_{2e+1}^{(J)}(n,k,q)\triangleq\label{eq23:A2ep1J}\db\\
&\left(\Sk_q(e+1,n)-\left(\frac{q^n}{q^k}-V_q(e,n)\right)
\left\lfloor\frac{n(q-1)}{e+1}\right\rfloor\right)\binom{2e+1}{e}^{-1}\emph{\cite[Theorem 4]{DavTombWordsMinWeght}}.\notag\db\\
&d=2e+2,~A_{2e+2}(n,k,q)\ge n\Psi_{2e+1}^{(J)}(n-1,k,q)(2e+2)^{-1}\emph{\cite[Corollary 4]{DavTombWordsMinWeght}}.\label{eq23:A2ep2JRed}
\end{align}
\textbf{\emph{(ii)}} The $[12,6,6]_3$ extended Golay code and all the codes obtained by puncturing it to an $[8,6,2]_3$ code achieve the bounds \eqref{eq23:A2ep2UP} and \eqref{eq23:A2ep2ReducOdd} \emph{\cite[p.\ 11]{DavTombWordsMinWeght}}.
\end{theorem}

\section{Some useful relations}\label{sec3:useful}
We consider linear $[n,k,d]_qR$ codes $C$ and the corresponding $m$-lifted $[n,k,d]_{q^m}R_m$ codes in accordance with Definition \ref{def1:lift}  and Proposition \ref{prop1:m-lifted}.

The following two lemmas are obvious.
\begin{lemma}\label{lem3:obvious1}
Let $C$ be a linear code, $C^\bot$ be its dual, and $C_m$ and $(C^\bot)_m$ be the corresponding  $m$-lifted codes. Then
 \begin{equation*}
   (C_m)^\bot=(C^\bot)_m,~ ((C^\bot)_m)^\bot= C_m.
 \end{equation*}
 \end{lemma}

\begin{lemma}\label{lem3:obvious2}
   Let $C$ be an $[n,k,d]_qR$ code and let $C_m$ be the corresponding  $m$-lifted $[n,k,d]_{q^m}R_m$ code. Let $S(C_{m})$ and $s(C_{m})$ be as in~\eqref{eq1:support non-zero weights}. Then we have the following:
 \begin{align*}
     &S(C_{m})\subseteq S(C_{m+1});~ S(C_{1})\subseteq S(C_{2})\subseteq S(C_{3})\subseteq\ldots\subseteq
     \{d,d+1,\ldots,n\};\db\\
   &   s(C_{m})\le s(C_{m+1});~s(C_{1})\le s(C_{2})\le s(C_{3})\ldots\le  n-d+1.
 \end{align*}

  If $S(C_{m^*})=\{d,d+1,\ldots,n\}$ then, for all $m\ge m^*$, we have $S(C_{m})=\{d,d+1,\ldots,n\}$, $s(C_{m})=n-d+1$.

  Let $A_{d+1}(C_m)=0$ for all $m\ge m^*$. Let $S(C_{m^*})=\{d,d+2,d+3,\ldots,n\}$. Then, for all $m\ge m^*$, we have $S(C_{m})=\{d,d+2,d+3,\ldots,n\}$, $s(C_{m})=n-d$.
 \end{lemma}

\begin{lemma}\label{lem3:NMDS}
  Let $C$ be an $[n,k,n-k]_q$ NMDS code, $C^\bot$ be its dual $[n,n-k,k]_q$ NMDS code, and let $C_m$ and $(C^\bot)_m$ be the corresponding $m$-lifted $[n,k,n-k]_{q^m}$ and $[n,n-k,k]_{q^m}$ codes. Then $C_m$ and $(C^\bot)_m$ are NMDS codes, too.
\end{lemma}

\begin{proof}
 The assertion follows from Definitions \ref{def1:lift}, \ref{def25:nearMDS}, and Propositions \ref{prop1:m-lifted}, \ref{prop24:m-lifted}.
\end{proof}

\begin{lemma}\label{lem3:minimWeight}
  Let $C$ be a linear $[n,k,d]_q$ code and let $C_m$ be the corresponding $m$-lifted code $[n,k,d]_{q^m}$, $m\ge1$, $C_1=C$, in accordance with Definition~\emph{\ref{def1:lift}} and Proposition~\emph{\ref{prop1:m-lifted}}. Let $A_d(C)$ and $A_d(C_m)$ be the number of  codewords of minimum weight~$d$ in the codes $C$ and $C_m$, respectively. Then $q-1$ divides $A_d(C)$ and
    \begin{equation}\label{eq3:minimWeight}
A_d(C_m)=\frac{A_d(C)}{q-1}(q^m-1),~m\ge1.
  \end{equation}
\end{lemma}

\begin{proof}
It is well known that if $\oc\in C$, $\T{wt}(\oc)=w$, then for any $\beta\in\F_q^*$, we have $\beta\oc\in C$, $\T{wt}(\beta\oc)=w$, that implies: $q-1$ divides $A_i(C)$ for any $i$. Now the assertions follow from Proposition~\ref{prop24:m-lifted}. In particular,
 the $A_d(C)$-set of weight $d$ codewords of $C$ can be partitioned into $A_d(C)/(q-1)$ disjoint $(q-1)$-subsets so that weight $d$ codewords $\oc_1,\oc_2$ belong to the same subset if and only if $\oc_1=\gamma\oc_2$ with $\gamma\in\F_q$. We form the $A_d(C)/(q-1)$-set $\As_d(C)$ containing one codeword from every subset. Finally, every
weight $d$ codeword of $C_m$ can be obtained as $\oc\beta$ with $\oc\in\As_d(C)$, $\beta\in\F_{q^m}^*$.
\end{proof}

\begin{theorem}\label{th3:LiftWeightDistr}
  Let $C$ be a linear $[n,k,d]_q=[n,k,n-k]_q$ NMDS code with $d=n-k$ and let $C_m$ be the corresponding NMDS $m$-lifted code $[n,k,d]_{q^m}=[n,k,n-k]_{q^m}$, $m\ge1$, $C_1=C$, in accordance with Definition~\emph{\ref{def1:lift}} and Proposition~\emph{\ref{prop1:m-lifted}}. Let $A_d(C)$ and $A_d(C_m)$ be the number of  codewords of minimum weight~$d$ in the codes $C$ and $C_m$, respectively. Then, for  $m\ge1$, $s\ge1$, the weight distribution of $C_m$ is as follows
    \begin{align}\label{eq3:weight distrib NMDS}
&A_{d+s}(C_m)=\binom{n}{k-s}\sum_{j=0}^{s-1}(-1)^j\binom{d+s}{j}(q^{m(s-j)}-1)+(-1)^s\binom{k}{s}\frac{A_d(C)}{q-1}(q^m-1).
\end{align}
In addition to above, let $A^{\max}_{d}(n,k,q)\triangleq \binom{n}{k-1}\frac{q-1}{k}$ as in \eqref{eq25:UpBndAd NMDS} and let $A_d(C)=A_{n-k}(C)= A^{\max}_{d}(n,k,q)$ that implies  $A_{d+1}(C)=0$. Then
  \begin{equation}\label{eq3:Adp1(Cm)=0}
    A_d(C_m)=A^{\max}_{d}(n,k,q^m)= \binom{n}{k-1}\frac{q^m-1}{k},~A_{d+1}(C_m)=0,~m\ge1.
    \end{equation}
\end{theorem}

\begin{proof}
 The assertions follow from Lemmas \ref{lem3:NMDS}, \ref{lem3:minimWeight}, and Theorem \ref{th25:NMDS}. In particular, for
$A_{d+1}(C_m)=0$ see Theorem \ref{th25:NMDS}(v)(vii).
\end{proof}

\begin{lemma}\label{lem3:LowBndAd}
Let $C$ be a linear $[n,k,d]_qR$ code of packing radius $e=\lfloor(d-1)/2\rfloor$, with $A_d(C)$ codewords of weight $d$.
Let $A_{d}(n,k,q)$ be $\min\limits_C A_d(C)$. Let $\Psi_{2e+2}(n,k,q)$ and $\Psi_{2e+1}^{(J)}(n,k,q)$ be as in \eqref{eq23:A2ep2UP} and in \eqref{eq23:A2ep1J}, respectively. The following lower bounds hold:
\begin{align}
&A_{2e+2}(n,k,q)\ge\left\lceil\Psi_{2e+2}(n,k,q)\right\rceil_{q-1}\db\label{eq3:A2ep2UP},\\
&\label{eq3:A2ep1RedCeilodd}
A_{2e+1}(n,k,q)\ge\left\lceil n\left\lceil\Psi_{2(e-1)+2}(n-1,k,q)\right\rceil_{q-1}(2e+1)^{-1}\right\rceil_{q-1},\db\\
&\label{eq3:A2ep1JonsonCeil}
A_{2e+1}(n,k,q)\ge\lceil\Psi_{2e+1}^{(J)}(n,k,q)\rceil_{q-1},\db\\
&\label{eq3:A2ep2JRedCeil}
A_{2e+2}(n,k,q)\ge\left\lceil n\left\lceil\Psi_{2e+1}^{(J)}(n-1,k,q)\right\rceil_{q-1}(2e+2)^{-1}\right\rceil_{q-1},
\end{align}
where $\lceil\bullet\rceil_{q-1}$ is the nearest greater integer multiple of $q-1$.
\end{lemma}

\begin{proof}
  The value $A_d(C)$ is an integer multiple of $q-1$, see Lemma \ref{lem3:minimWeight}. Now the assertions
  \eqref{eq3:A2ep2UP}--\eqref{eq3:A2ep2JRedCeil} follow from \eqref{eq23:A2ep2UP}--\eqref{eq23:A2ep2JRed}.
\end{proof}

\begin{lemma} \label{lem3:Rierarh}
We consider the $[n,k,d]_{q^m}R_m$ code $C_m$ over the field
$\F_{q^m}$, obtained by $m$-lifting an $[n,k,d]_qR$ code $C$  over $\F_q$. We have
\begin{description}
  \item[(i)] \emph{\cite[Section III, equations (1), (2)]{GCR2021IEEE-EFS}}, \emph{\cite[Section II]{GCR2022IEEE-EWS}}
  \begin{align}\label{eq3:Rm}
    &m\le R_m\le mR_1;~R_1\le R_2\le\ldots\le R_{n-k}=R_{n-k+1} =R_{n-k+2}=\ldots =n-k.
  \end{align}

  \item[(ii)]If $R_m=n-k$ with $m=m'\le n-k$ then $R_m=n-k$ for all $m\ge m'$.
  \end{description}
\end{lemma}

\begin{proof}
\begin{description}
  \item[(i)]
  In \cite[Section III]{GCR2021IEEE-EFS},\cite[Section II.A]{GCR2022IEEE-EWS},
the so-called $m$-th generalized covering radius of an $[n,k,d]_qR$ code $C$ over the field $\F_q$ is considered and it is shown that this generalized radius is equal to the regular covering radius $R_m$ of the $m$-lifted code $C_m$. The results of \cite{GCR2021IEEE-EFS,GCR2022IEEE-EWS} give rise to the assertion.

\item[(ii)] The assertion follows from the case (i) of this lemma. \qedhere
  \end{description}
\end{proof}

\begin{theorem}\label{th3:NMDS-Adp1=0-UPWS} \textbf{\emph{(general theorem on infinite families of  UPWS NMDS codes)}}
Let $C_m$ be the $[n,k,n-k]_{q^m}R_m$ NMDS code over the field
$\F_{q^m}$, obtained by $m$-lifting an $[n,k,n-k]_qR_1$ NMDS code $C_1$  over $\F_q$, $m\ge 1$.  Let
\begin{equation}\label{eq3:AdC1main}
  A_{d}(C_1)=A_{n-k}(C_1)=\binom{n}{k+1}\frac{q-1}{n-k}.
\end{equation}
We have the following.
  \begin{description}
    \item[(i)]  The codes $C_m=[n,k,n-k]_{q^m}R_m$  with $m\ge n-k$ form an infinite family of UPWS NMDS codes with $R_m=n-k$.

    \item[(ii)]  If $R_{m'}=n-k$ for $m'< n-k$ then the codes $C_m$  with $m\ge m'$ form an infinite family of UPWS NMDS codes with $R_m=n-k$.
  \end{description}
\end{theorem}

\begin{proof}
  \begin{description}
    \item[(i)]
  By Definition \ref{def25:nearMDS}(ii), the code $C_m^\bot$ is an $[n,n-k,k]_{q^m}$ NMDS code, $m\ge 1$. So, $d(C_m)=n-k$, $d(C_m^\bot)=k$, $m\ge 1$.  By \eqref{eq25:UpBndAd NMDS}, \eqref{eq3:AdC1main}, and Theorem \ref{th25:NMDS}(iv),
  \begin{equation*}
    A^{\max}_{d}(n,n-k,q)=\binom{n}{k+1}\frac{q-1}{n-k}=A_{d}(C_1)=A_{d}(C_1^\bot),
  \end{equation*}
 that implies $A_{d+1}(C_1^\bot)=0$, see \eqref{eq25:UpBndAd NMDS}. Therefore,  by
 Theorem \ref{th25:NMDS}(vii), $A_{d+1}(C_m^\bot)=0$ for all $m\ge 1$. This means, $s(C_m^\bot)\le n-d(C_m^\bot)=n-k$, $m\ge 1$, see Lemma~\ref{lem3:obvious2}.

  For $C_m$, by Lemma \ref{lem3:Rierarh}(i), covering radius $R_m=n-k$ if $m\ge n-k$. By above and by Delsarte bound, Theorem \ref{th22:cosets}(i), $R_m=n-k\le s(C_m^\bot)\le n-k$. The only possibility is $R_m=n-k=s(C_m^\bot)$. Now, the assertion follows from Theorem \ref{th1:UPWS nes-suf}.
      \item[(ii)]  The assertion follows adding Lemma \ref{lem3:Rierarh}(ii) to the proof of the case (i). \qedhere
  \end{description}
\end{proof}

\section{New infinite families of near-MDS codes uniformly packed in the wide sense (UPWS)} \label{sec4:newUPWS}

\subsection{The ternary Golay code, its dual and extended one, and the corresponding lifted codes}

We consider the ternary perfect $[11,6,5]_32$ Golay code $\G_{11}$, its dual $[11,5,6]_35$ code $\G_{11}^\bot$, and
the self-dual quasi-perfect $[12,6,6]_33$ extended ternary Golay code $\G_{12}\triangleq\widehat{\G}_{11}$. Also, we research
the corresponding $m$-lifted codes $\G_{11,m}=[11,6,5]_{3^m}R_{11,m}$, $\G_{11,m}^\bot=[11,5,6]_{3^m}R_{11,m}^\bot$, $\G_{12,m}=[12,6,6]_{3^m}R_{12,m}$, $m\ge2$, see Definition \ref{def1:lift}, Proposition \ref{prop1:m-lifted}.
It is known that $\G_{11}$, $\G_{11}^\bot$,  and $\G_{12}$ are UPWS codes \cite {BasZaiZinUP,BasZin,BorgRiZin2019PIT} and, moreover,
$\G_{11}$  and $\G_{12}$ are CR codes, see Proposition \ref{prop23:examplesCR}(iii). Also, see \eqref{eq1:support non-zero weights}, $S(\G_{11})=\{5,6,8,9,11\}$, $s(\G_{11})=5$, $S(\G_{11}^\bot)=\{6,9\}$, $s(\G_{11})=2$, $S(\G_{12})=\{6,9,12\}$, $s(\G_{12})=3$  \cite{HufPless,MWS,Roth}. Let $I_k$ be the $k\times k$ identity matrix. The generator matrices of the noted codes are as follows:
\begin{align}\label{eq41:Golay3}
& G(\G_{11})=[I_6\,|\,M(\G_{11})],~ G(\G_{11}^\bot)=[I_5\,|\,M(\G_{11}^\bot)],
~G(\G_{12})=[I_6\,|\,M(\G_{12})];\db\\
&  M(\G_{11})= \left[ \begin{array}{c}
   2  2  2  2  2 \\
   2  2  1  1  0 \\
   2  1  2  0  1 \\
   1  2  0  2  1  \\
   1  0  2  1  2  \\
   0  1  1  2  2
    \end{array}  \right],~
 M(\G_{11}^\bot)=\left[ \begin{array}{c}
    0 2 2 2 2 2  \\
    2 2 1 0 2 1  \\
    2 2 0 1 1 2  \\
    2 1 1 2 0 2   \\
    2 1 2 1 2 0
   \end{array}  \right],~
    M(\G_{12})= \left[ \begin{array}{c}
      0  1  1  1  1  1 \\
      1  0  1  2  2  1 \\
      1  1  0  1  2  2 \\
      1  2  1  0  1  2  \\
      1  2  2  1  0  1  \\
      1  1  2  2  1  0
       \end{array}  \right]. \notag
       \end{align}

    \begin{proposition}\label{prop41:Golay}
   For the $m$-lifted codes  $\G_{11,m}=[11,6,5]_{3^m}R_{11,m}$,\\
   $\G_{11,m}^\bot =[11,5,6]_{3^m}R_{11,m}^\bot$,
     and $\G_{12,m}=[12,6,6]_{3^m}R_{12,m}$, $m\ge2$, the following hold.
   \begin{description}
    \item[(i)] All these codes are NMDS.

    \item[(ii)] For covering radii we have the following: $R_{11,2}=R_{11,3}=4$, $R_{11,m}=5, ~m\ge4$; $R_{11,2}^\bot=R_{11,3}^\bot=5$,
 $R_{11,m}^\bot=6,~m\ge4$; $R_{12,2}=4,R_{12,3}=5$, $R_{12,m}=6, ~m\ge6$.

    \item[(iii)] We have  the following:
    \begin{align}
    & S(\G_{11,m})=\{5,6,7,8,9,10,11\},~  s(\G_{11,m})=7,~ m\ge2;\label{eq41:s11m}\db\\
     & S(\G_{11,m}^\bot )=\{6,8,9,10,11\}, ~ s(\G_{11,m}^\bot )=5,~m\ge2;\label{eq41:s11m bot}\db\\
     &   S(\G_{12,m})=\{6,8,9,10,11,12\}, ~ s(\G_{12,m})=6,~m\ge2;\label{eq41:s12m}\db\\
    &\label{eq41:Adp1=a7=0}
      A_{d+1}(\G_{11,m}^\bot )=A_7(\G_{11,m}^\bot )=A_{d+1}(\G_{12,m})=A_7(\G_{12,m})=0,~m\ge1.
    \end{align}
       \end{description}
  \end{proposition}

  \begin{proof}
    \begin{description}
    \item[(i)]
    The assertion follows from Definition \ref{def25:nearMDS} and Lemma \ref{lem3:NMDS}.

    \item[(ii)] By
    Lemma \ref{lem3:Rierarh}(i) with \eqref{eq3:Rm}, we have the following: $R_{11,m}=5, ~m\ge5$;  $R_{11,m}^\bot=R_{12,m}=6, ~m\ge6$. In addition, from \eqref{eq41:Golay3}, using the  symbolic computation system Magma \cite{Magma}, we obtain: $R_{11,2}=R_{11,3}=4$, $R_{11,4}=5$;  $R_{11,2}^\bot=R_{11,3}^\bot=5$, $R_{11,4}^\bot=R_{11,5}^\bot=6$; $R_{12,2}=4$, $R_{12,3}=5$.

    \item[(iii)] By the case (i), we can use Theorem \ref{th3:LiftWeightDistr} with \eqref{eq3:weight distrib NMDS}.

       $\bullet$ By \eqref{eq23:A2ep2ReducOdd} or \eqref{eq23:A2ep1J}, $A_d(\G_{11})=A_5(\G_{11})=132$.  By \eqref{eq3:weight distrib NMDS} with $m=2$, we obtain the weight distribution of the non-zero codewords (see also \cite[Example 8]{DingTangDesignsNMDS2020IEEE})
    \begin{align*}
     &\{A_5=A_6=528, A_7=15840, A_8=40920, A_9=129800,\db\\
     & A_{10}=198000,A_{11}=145824\}\T{ for }\G_{11,2}=[11,6,5]_{9}4.\notag
    \end{align*}
    This implies $S(\G_{11,2})=\{5,6,7,8,9,10,11\}$, $s(\G_{11,2})=7$. Now, the assertion \eqref{eq41:s11m} for $\G_{11,m}$ follows from Lemma \ref{lem3:obvious2}.

    $\bullet$ By \eqref{eq25:UpBndAd NMDS}, $A^{\max}_{d}(n,k,q)= \binom{n}{k-1}\frac{q-1}{k}$. By \eqref{eq3:A2ep2JRedCeil} or by Theorem \ref{th25:NMDS}(v) with \eqref{eq25:AdNMDS}, $A_d(\G_{11}^\bot)=A_6(\G_{11}^\bot)=132$. By \eqref{eq25:AdNMDS} or \eqref{eq23:A2ep2UP}, or \eqref{eq3:A2ep2JRedCeil}, $A_d(\G_{12})=A_6(\G_{12})=264$. For $m\ge2$, by \eqref{eq3:minimWeight},
   we have $A_{d}(\G_{11,m}^\bot)=A_{6}(\G_{11,m}^\bot)=132(3^m-1)/(3-1),~A_{d}(\G_{12,m})=A_{6}(\G_{12,m})=264(3^m-1)/(3-1)$,
 that implies, see \eqref{eq3:Adp1(Cm)=0}, $A_{d}(\G_{11,m}^\bot)=A_{6}(\G_{11,m}^\bot)=
 A^{\max}_{d}(11,5,3^m)$, $A_{d}(\G_{12,m})=A_{6}(\G_{12,m})=A^{\max}_{d}(12,6,3^m)$. This proves \eqref{eq41:Adp1=a7=0}.

$\bullet$ By \eqref{eq3:weight distrib NMDS}, using $A_{d}(\G_{11}^\bot)$ and $A_{d}(\G_{12})$, for $m=2$ we obtain the  weight distributions:
\begin{align*}
 &\{A_6=528, A_8=7920, A_9=11000, A_{10}=23760, A_{11}=15840\}\T{ for }\G_{11,2}^\bot;\db\\
 &\{\A_6=1056,A_8=23760, A_9=44000, A_{10}=142560,A_{11}=190080,A_{12}=129984\}\db\\
 &\T{ for }\G_{12,2}=[12,6,6]_{9}4.
\end{align*}
  This implies $S(\G_{11,2}^\bot)=\{6,8,9,10,11\}$, $s(\G_{11,2}^\bot)=5$,
$S(\G_{12,2})=\{6,8,9,10,11,12\}$, $s(\G_{12,2})=6$.
Now, the assertions \eqref{eq41:s11m bot}, \eqref{eq41:s12m}, follow from Lemma \ref{lem3:obvious2} and \eqref{eq41:Adp1=a7=0}.
\qedhere
    \end{description}
  \end{proof}

\begin{theorem}\label{th41:UPWSliftGolay}
\emph{(The main results on the $m$-lifted codes $\G_{11,m},\G_{12,m},\G_{11,m}^\bot$, $m\ge1$)}
\begin{description}
    \item[(i)] We have two infinite families $\Ds_1$, $\Ds_2$ of UPWS NMDS codes.\\
   $\Ds_1$: the $m$-lifted extended Golay codes,\\
    \phantom{$\Ds_1$: }$\G_{12,1}=[12,6,6]_{3}3$, $\G_{12,m}=[12,6,6]_{3^m}6$, $m\ge6$.\\
  $\Ds_2$: the $m$-lifted Golay codes, $\G_{11,1}=[11,6,5]_{3}2$, $\G_{11,m}=[11,6,5]_{3^m}5$, $m\ge4$.

  \item[(ii)]
  The following NMDS codes are \emph{not} UPWS:\\
   -- the infinite family of $m$-lifted dual Golay codes $\G_{11,m}^\bot=[11,5,6]_{3^m}R_{11,m}^\bot$, $m\ge2$;\\
   -- the two $m$-lifted Golay codes $\G_{11,m}=[11,6,5]_{3^m}4$, $m=2,3$;\\
   -- the two $m$-lifted extended Golay codes  $\G_{12,2}=[12,6,6]_{3^2}4$, $\G_{12,3}=[12,6,6]_{3^3}5$.

  \item[(iii)] The dual Golay code $\G_{11}^\bot=\G_{11,1}^\bot=[11,5,6]_35$ is CR. Also, this code, punctured on the position $11$, is the  $(\G_{11}^\bot)^*_{11}=[10,5,5]_34$ self-dual CR code. Moreover, surprisingly, the code \emph{\cite[p.\ 24, (S.13)]{BorgRiZin2019PIT}}, obtained as the set of all codewords of weights $0,6,9$ of the Golay code $\G_{11}=[11,6,5]_32$, coincides with $\G_{11}^\bot$ which is the dual code of $\G_{11}$.

  \item[(iv)]  Let $A^{\max}_{d}(n,k,q)=\binom{n}{k-1}\frac{q-1}{k}$ be as in  \eqref{eq25:UpBndAd NMDS}. For $m\ge1$, all the codes of the infinite families $\G_{12,m}=[n,k,d]_q=[12,6,6]_{3^m}$  and $\G_{11,m}^\bot=[n,k,d]_q=[11,5,6]_{3^m}$ have the maximum possible (for NMDS codes) number of codewords of minimal weight $d=6$ which is equal to $A^{\max}_{6}(12,6,3^m)$ and $A^{\max}_{6}(11,5,3^m)$, respectively. Moreover, for $m=1$, this maximal number coincides with the minimal possible (for non-restricted codes) number of codewords of minimal weight $d=6$.

      Also, the Golay code $\G_{11}=[11,6,5]_{3}$ has the minimal possible (for non-restricted codes) number of codewords of minimal weight $d=5$.

  \item[(v)]  For $m\ge1$, all the codes of the infinite families $\G_{12,m}=[n,k,d]_q=[12,6,6]_{3^m}$  and $\G_{11,m}^\bot=[n,k,d]_q=[11,5,6]_{3^m}$ have no codewords of weight $d+1=7$.
\end{description}
\end{theorem}

\begin{proof}
  \begin{description}
    \item[(i)]
  By Proposition \ref{prop41:Golay}(ii) and \eqref{eq41:s12m},
  for the self-dual $\G_{12,m}=[12,6,6]_{3^m}R_{12,m}$ codes with $m\ge6$, we have $R_{12,m}=s(\G_{12,m})=6$. By Proposition \ref{prop41:Golay}(ii) and \eqref{eq41:s11m bot}, for the $\G_{11,m}=[11,6,5]_{3^m}R_{11,m}$ codes with $m\ge4$, we have $R_{11,m}=s(\G_{11,m}^\bot)=5$. So, by Theorem \ref{th1:UPWS nes-suf}, these families consist of UPWS codes. All these codes are NMDS, see Proposition \ref{prop41:Golay}(i). Also, we can use directly Theorem \ref{th3:NMDS-Adp1=0-UPWS}. Finally, by
  \cite{BasZaiZinUP,BasZin,BorgRiZin2019PIT}, $\G_{11,1}$ and $\G_{12,1}$ are UPWS codes.

  \item[(ii)]
  By Proposition \ref{prop41:Golay}(ii) and \eqref{eq41:s11m}, for the $\G_{11,m}^\bot=[11,5,6]_{3^m}R_{11,m}^\bot$ codes with $m\ge2$, we have $R_{11,m}^\bot=5<s(\G_{11,m})=7$ if $m=2,3$, and $R_{11,m}^\bot=6<s(\G_{11,m})=7$ if $m\ge4.$ So, by Theorem \ref{th1:UPWS nes-suf}, this  family does not contain any UPWS code.

  By Proposition \ref{prop41:Golay}(ii) and \eqref{eq41:s11m bot}, \eqref{eq41:s12m}, we have $R_{11,m}=4<s(\G_{11,m}^\bot)=5$, $m=2,3$;
   $R_{12,2}=4<s(\G_{12,2})=6$; $R_{12,3}=5<s(\G_{12,3})$=6.
   So, by Theorem \ref{th1:UPWS nes-suf}, the corresponding codes $\G_{11,2}$, $\G_{11,3}$, $\G_{12,2}$, $\G_{12,3}$ are not UPWS.

  \item[(iii)] By computer search with Magma \cite{Magma}, for the code $\G_{11}^\bot=\G_{11,1}^\bot=[11,5,6]_35$, all the cosets of each weight 1,2,3,4,5  have identical weight distribution.
      We have similar results for the $(\G_{11}^\bot)^*_{11}=[10,5,5]_34$ code. Now the first two assertions follow from Definition \ref{def22:t regular CR}. The last assertion has been checked directly by computer.

  \item[(iv)] The assertions follow from \eqref{eq23:A2ep2UP}, \eqref{eq23:A2ep1J}, \eqref{eq3:A2ep2JRedCeil}, and  Proof of Proposition \ref{prop41:Golay}(iii).

  \item[(v)]  The assertion follows from \eqref{eq41:Adp1=a7=0}.
       \qedhere
    \end{description}
 \end{proof}


\subsection{The punctured ternary Golay codes, their dual ones, and the corresponding lifted codes}

Let $\G_{11}$ be the $[11,6,5]_32$ Golay code given by the generator matrix $G(\G_{11})=[I_6\,|\,M(\G_{11})]$  \eqref{eq41:Golay3}. We consider the codes, obtained by puncturing $\G_{11}$ on the positions 11,  \{11,10\}, and \{11,10,9\}, and the corresponding dual codes. We denote these codes $\G_{10}\triangleq(\G_{11})^*_{11}$, $\G_{9}\triangleq(\G_{11})^*_{11,10}$,  $\G_{8}\triangleq(\G_{11})^*_{11,10,9}$, and $\G_{10}^\bot$, $\G_{9}^\bot$, $\G_{8}^\bot$. The generator matrices $G(\G_n)$ of the punctured codes are formed by removing the corresponding columns from $G(\G_{11})$~\eqref{eq41:Golay3}; for the dual codes we use the standard methods of \cite{HufPless,MWS,Roth}.  We have
\begin{align}\label{eq42:Golay3Punct}
 & M(\G_{10})= \left[ \begin{array}{c}
   2  2  2  2 \\
   2  2  1  1 \\
   2  1  2  0 \\
   1  2  0  2  \\
   1  0  2  1  \\
   0  1  1  2
    \end{array}  \right],~
     M(\G_{9})= \left[ \begin{array}{c}
   2  2  2 \\
   2  2  1 \\
   2  1  2 \\
   1  2  0 \\
   1  0  2 \\
   0  1  1
    \end{array}  \right],~
     M(\G_{8})= \left[ \begin{array}{c}
   2  2 \\
   2  2 \\
   2  1 \\
   1  2  \\
   1  0  \\
   0  1
    \end{array}  \right];\db\\
&G(\G_{n})=[I_6\,|\,M(\G_{n})],~G(\G_{n}^\bot)=[-M(\G_{n})^{tr}\,|\,I_{n-6} ],~n=10,9,8,\notag
    \end{align}
where $I_k$ is the $k\times k$ identity matrix, $tr$ is the sign of transposition.

Also, we investigate the corresponding $m$-lifted codes $\G_{n,m}$, $\G_{n,m}^\bot$, $n=10,9,8$, $m\ge2$, see Definition \ref{def1:lift}, Proposition \ref{prop1:m-lifted}. For $m=1$, we put $\G_{n,1}=\G_{n}$, $\G_{n,1}^\bot=\G_{n}^\bot$.

\begin{proposition}\label{prop42:codesG1G1botproperty}
For the punctured codes $\G_{10}\triangleq(\G_{11})^*_{11}$, $\G_{9}\triangleq(\G_{11})^*_{11,10}$,  $\G_{8}\triangleq(\G_{11})^*_{11,10,9}$, and their dual $\G_{10}^\bot$, $\G_{9}^\bot$, $\G_{8}^\bot$, given by \eqref{eq42:Golay3Punct}, the following hold.
\begin{description}
  \item[(i)]The codes $\G_{n},~\G_{n}^\bot$, are NMDS with parameters $\G_{10}=[10, 6, 4]_32$, $\G_{9} =[9, 6, 3]_32$, $\G_{8} =[8, 6, 2]_31$;~
    $\G_{10}^\bot=[10,4,6]_35,~\G_{9}^\bot=[9,3,6]_35,~\G_{8}^\bot=[8,2,6]_35$.

  \item[(ii)] On the number of small weights, for $n=10,9,8$, we have the following:
\begin{align}
   &A_{d+1}(\G_{n}^\bot)=A_{7}(\G_{n}^\bot)=0,~A_{d+2}(\G_{n}^\bot)=A_{8}(\G_{n}^\bot)=0;\label{eq42:SmallWeights}\db\\
     &A_d(\G_{n})=A_{n-6}(\G_{n})=A_d(\G_{n}^\bot)=A_6(\G_{n}^\bot)=\binom{n}{n-7}\frac{3-1}{n-6}
     =A^{\max}_{d}(n,n-6,3).\notag
         \end{align}

  \item[(iii)]  For a code $C$, let $S(C)$ and $s(C)$ be as in \eqref{eq1:support non-zero weights}. Then
\begin{align*}
    & S(\G_{n})=\{n-6,n-5,\ldots,n\},~ s(\G_{n})=7,~n=10,9,8;\db\\
    &  S(\G_{n}^\bot )=\{6,9\},~s(\G_{n}^\bot)=2,~n=10,9;~S(\G_{8}^\bot)=\{6\},~s(\G_{8}^\bot)=1.
 \end{align*}
  \item[(iv)]  The codes $\G_{10},~\G_{9}$, and $\G_{8}$ are CR (and hence, UPWS).

  \item[(v)]  The code $\G_{10}^\bot=[10,4,6]_35$ is $2$-regular.
\end{description}
\end{proposition}

\begin{proof}
\begin{description}
  \item[(i first part)]
The code $\G_{11}=[11,6,5]_32$ is NMDS. Therefore, by Theorem \ref{th25:NMDS}(iii), the punctured codes are NMDS such that $\G_{n}=[n, 6, n-6]_3$. Hence, the dual codes $\G_{n}^\bot=[n,n-6,6]_3$ are NMDS as well, see Theorem \ref{th25:NMDS}(i). Also, for
$\G_{n}^\bot$ we can use Theorem \ref{th25:NMDS}(ii). The case (i) has been proved except the values of covering radii.

  \item[(ii)] By Theorem \ref{th25:NMDS}(vi), $\G_{n}^\bot=[n,n-6, 6]_3$, $n=10,9,8$, are extremal NMDS codes with $A_{d+1}(\G_{n}^\bot)=A_{7}(\G_{n}^\bot)=0$, that implies, by Theorem \ref{th25:NMDS}(v) and \eqref{eq25:UpBndAd NMDS},
  $A_{d}(\G_{n}^\bot)=A_{6}(\G_{n}^\bot)=A^{\max}_{d}(n,n-6,3)=\binom{n}{n-7}\frac{3-1}{n-6}$.  By Theorem \ref{th25:NMDS}(iv), $A_{d}(\G_{n})=A_{d}(\G_{n}^\bot)$. We have proved the 1-st and 3-rd equalities in \eqref{eq42:SmallWeights}. The 2-nd one holds by \eqref{eq25:Adp2 NMDS}.

   Also, for $A_{d}(\G_{10})$ we can use \eqref{eq25:AdNMDS}, and for $A_{d}(\G_{n}^\bot)$, $n=10,9,8$, we can apply Theorem \ref{th23:Ad low bnds}(ii) together with \eqref{eq23:A2ep2UP}, \eqref{eq23:A2ep2ReducOdd}.

\item[(iii)]
By the case (i), we can use \eqref{eq3:weight distrib NMDS}.  From \eqref{eq42:SmallWeights} we take $A_{d}(\G_{n})=A_{d}(\G_{n}^\bot)$, $A_{7}(\G_{n}^\bot)=A_{8}(\G_{n}^\bot)=0$. Then, by \eqref{eq3:weight distrib NMDS} with $m=1$, we obtain the weight distributions:
       \begin{align*}
       &\{A_4=60, A_5=144, A_6=60, A_7=240, A_8=180, A_9=20, A_{10}=24\}\T{ for } \G_{10};\db\\
       & \{A_3=24, A_4=A_5=108, A_6=192, A_7=216, A_8=54, A_9=26\}\T{ for }\G_{9};\db\\
       & \{A_2= 8, A_3=64, A_4=120, A_5=176, A_6= 232, A_7=96, A_8=32 \}\T{ for }\G_{8};\db\\
       &\{A_6=60, A_9=20\}\T{ for }\G_{10}^\bot;~\{A_6=24, A_9=2\}\T{ for }\G_{9}^\bot;~\{A_6=8\}\T{ for }\G_{8}^\bot.
       \end{align*}

  \item[(iv)(i second part)] By the case (iii), the codes $\G_{10}^\bot$, $\G_{9}^\bot$, are two-weight ones. So, by Proposition \ref{prop23:examplesCR}(ii), their dual codes $\G_{10}$, $\G_{9}$, are CR with covering radius 2. Also, for $\G_{10}$, we can use Proposition \ref{prop23:examplesCR}(iii).

By \eqref{eq42:Golay3Punct}, for $\G_{8}$ we can represent a parity check matrix $H(\G_{8})$ as follows
\begin{equation*}
H(\G_{8})  = G(\G_{8}^\bot)=\left[ \begin{array}{ccc}
0111&|&0212\\
2012&|&1011
\end{array}  \right].
\end{equation*}
 It is formed by two parity check matrices of the $[4,2,3]_31$ Hamming code. By
Proposition \ref{prop23:examplesCR}(iv), $\G_{8}$ is a CR code with distance 2 and covering radius 1.

We have proved the case (iv) and the values of covering radii for $\G_{n}$, $n=10,9,8$.

  \item[(v)(i third part)] By computer search with the help of  Magma \cite{Magma}, we obtained that all the codes $\G_{n}^\bot$, $n=10,9,8$, have covering radius $R=5$. Moreover, as a byproduct, the search showed that for the code $\G_{10}^\bot=[10,4,6]_35$, all the cosets of weight 1 (similarly, all the ones of weight 2) have identical weight distribution. Now the assertion of the case~(v) follows from Definition \ref{def22:t regular CR}.
\qedhere
\end{description}
\end{proof}

\begin{proposition}\label{prop42:Golay}
   Let $n=10,9,8$, $m\ge2$. For the $m$-lifted codes\\ $\G_{n,m}=[n,6,n-6]_{3^m}R_{n,m}$,~
   $\G_{n,m}^\bot=[n,n-6,6]_{3^m}R_{n,m}^\bot$, the following hold.
   \begin{description}
    \item[(i)] All these codes are NMDS.

    \item[(ii)] The covering radii are as follow. $
      R_{10,2}=3,~R_{10,m}=4,~m\ge3$; $R_{9,m}=3,~m\ge2$;
            $R_{8,m}=2,~m\ge2$; $R_{n,m}^\bot=6,~m\ge2, ~n=10,9,8.$

    \item[(iii)]
    On the number of small weights we have the following:
\begin{align}
   &A_{d+1}(\G_{n,m}^\bot)=A_{7}(\G_{n,m}^\bot)=0,~n=10,9,8,~m\ge2;\label{eq42:small weights m2}\db\\
     &A_d(\G_{n,m})=A_{n-6}(\G_{n,m})=A_d(\G_{n,m}^\bot)=A_6(\G_{n,m}^\bot)=\binom{n}{n-7}\frac{3^m-1}{n-6}
     \label{eq42:small weights2-m2}\db\\
     &=A^{\max}_{d}(n,n-6,3^m),~n=10,9,8,~m\ge2.\notag
         \end{align}

    \item[(iv)] For a code $C$, let $S(C)$ and $s(C)$ be as in \eqref{eq1:support non-zero weights}. For $m\ge2$, the following hold:
    \begin{align}\label{}
    & S(\G_{n,m})=\{n-6,n-5,\ldots,n\},~ s(\G_{n,m})=7,~n=10,9,8;\label{eq42:Ssnm}\db\\
    &S(\G_{10,m}^\bot )=\{6,8,9,10\},~s(\G_{10,m}^\bot )=4;
    ~S(\G_{9,m}^\bot )=\{6,8,9\},~s(\G_{9,m}^\bot)=3;\label{eq42:Ssnm bot}\db\\
    &S(\G_{8,m}^\bot)=\{6,8\},~s(\G_{8,m}^\bot)=2.\label{eq42:s8m bot ge2}
    \end{align}
  \end{description}
  \end{proposition}

  \begin{proof}
      \begin{description}
    \item[(i)]
    The assertion follows from Definition \ref{def25:nearMDS} and Lemma \ref{lem3:NMDS}.

    \item[(ii)] By  \eqref{eq3:Rm}, we have the following:
    $R_{10,m}=4,~m\ge4$;  $R_{9,m}=3,~m\ge3$;  $R_{8,m}=2,~m\ge2$;  $R_{n,m}^\bot=6, ~m\ge6,~n=10,9,8$.
    The  values $R_{10,2}=3,~R_{10,3}=4$, $R_{9,2}=3$, $R_{n,m}^\bot=6,~m=2,3,4,5,~n=10,9,8$, have been obtained using Magma \cite{Magma}.

    \item[(iii)] The assertions follow from \eqref{eq42:SmallWeights}, Theorems \ref{th25:NMDS}(vii), \ref{th3:LiftWeightDistr}, and Lemma \ref{lem3:minimWeight}.

    \item[(iv)]
    The assertion \eqref{eq42:Ssnm} follows from Proposition \ref{prop42:codesG1G1botproperty}(iii) and Lemma
    \ref{lem3:obvious2}.

    By the case (i), we can use \eqref{eq3:weight distrib NMDS}. We take $A_{d}(\G_{n,1}^\bot)$ from \eqref{eq42:SmallWeights}.  Then, by \eqref{eq3:weight distrib NMDS} with $m=2$,
we obtain the following weight distributions of the non-zero codewords:
 \begin{align*}
  & \{A_6=240,  A_8=2160, A_9=2000, A_{10}=2160\}\T{ for }\G_{10,2}^\bot;\db\\
  &\{A_6=96,  A_8=432, A_9=200\}\T{ for }\G_{9,2}^\bot;~\{A_6=32,  A_8=48\}\T{ for }\G_{8,2}^\bot.
 \end{align*}
   So,
$S(\G_{10,2}^\bot)=\{6,8,9,10\}$, $s(\G_{10,2}^\bot)=4$, $S(\G_{9,2}^\bot)=\{6,8,9\}$, $s(\G_{9,2}^\bot)=3$,  $S(\G_{8,2}^\bot)=\{6,8\}$, $s(\G_{6,2}^\bot)=2$.
Now \eqref{eq42:Ssnm bot} \eqref{eq42:s8m bot ge2} follow from Lemma \ref{lem3:obvious2} and \eqref{eq42:small weights m2}.
    \qedhere
    \end{description}
  \end{proof}

 \begin{theorem}\label{th42:UPWSliftGolay}
 \emph{(The main results on the $m$-lifted codes $\G_{n,m},\G_{n,m}^\bot$, $n=10,9,8$, $m\ge1$)}
\begin{description}
      \item[(i)] We have three infinite families $\Ds_3$, $\Ds_4$, $\Ds_5$ of UPWS NMDS codes.\\
  $\Ds_3$: the $m$-lifted Golay codes, punctured on the position $11$;\\
   \phantom{$\Ds_3$: }$\G_{10,1}=[10,6,4]_{3}2$, $\G_{10,m}=[10,6,4]_{3^m}4,~m\ge3$.\\
  $\Ds_4$: the $m$-lifted Golay codes, punctured on the positions $11,10$;\\
   \phantom{$\Ds_4$: }$\G_{9,1}=[9,6,3]_{3}2$, $\G_{9,m}=[9,6,3]_{3^m}3,~m\ge2$.\\
  $\Ds_5$: the $m$-lifted Golay codes, punctured on the positions $11,10,9$;\\
   \phantom{$\Ds_5$: }$\G_{8,1}=[8, 6, 2]_{3}1$, $\G_{8,m}=[8, 6, 2]_{3^m}2,~m\ge2$.

  \item[(ii)]
The codes of the family
$\Ds_5$, are CR NMDS. Also, the code $\G_{9,1}=[9,6,3]_{3}2$ is CR NMDS;
the code $\G_{10}^\bot=[10,4,6]_35$ is $2$-regular.

  \item[(iii)]
  The following NMDS codes are \emph{not} UPWS.\\
    --  The three infinite families, consisting of codes dual to punctured Golay ones and the corresponding $m$-lifted codes; in details:\\
  $\G_{10,1}^\bot=[10,4,6]_{3}5$, $\G_{10,m}^\bot=[10,4,6]_{3^m}6$, $m\ge2$;\\
  $\G_{9,1}^\bot=[9,3,6]_{3}5$, $\G_{9,m}^\bot=[9,3,6]_{3^m}6$, $m\ge2$;\\
  $\G_{8,1}^\bot=[8,2,6]_{3}6$, $\G_{8,m}^\bot=[8,2,6]_{3^m}6$, $m\ge2$.\\
  -- The code $\G_{10,2}=[10,6,4]_{3^2}3$.

  \item[(iv)]  Let $A^{\max}_{d}(n,k,q)=\binom{n}{k-1}\frac{q-1}{k}$ be as in  \eqref{eq25:UpBndAd NMDS}. For $m\ge1$, all the codes of the three infinite families $\G_{10,m}^\bot=[10,4,6]_{3^m}$, $\G_{9,m}^\bot=[9,3,6]_{3^m}$, and
  $\G_{8,m}^\bot=[8,2,6]_{3^m}$, have the maximum possible (for NMDS codes) number of codewords of minimal weight $d=6$ which is equal to $A^{\max}_{6}(10,6,3^m)$, $A^{\max}_{6}(9,6,3^m)$, and $A^{\max}_{6}(8,6,3^m)$, respectively.

      Also, the punctured Golay codes $\G_{10}=[10,6,4]_{3}$, $\G_{9}=[9,6,3]_{3}$, and
  $\G_{8}=[8,6,2]_{3}$, have the minimal possible (for non-restricted codes) number of codewords of minimal weight $d=6$.

  \item[(v)]  For $m\ge1$, all the codes of the three infinite families $\G_{10,m}^\bot=[10,4,6]_{3^m}$, $\G_{9,m}^\bot=[9,3,6]_{3^m}$, and
  $\G_{8,m}^\bot=[8,2,6]_{3^m}$, have no codewords of weight $d+1=7$.

    \end{description}
\end{theorem}

\begin{proof}
    \begin{description}
    \item[(i)]
  By Proposition \ref{prop42:Golay}(ii) and  \eqref{eq42:Ssnm bot}, \eqref{eq42:s8m bot ge2}, for $\G_{10,m}=[10,6,4]_{3^m}R_{10,m}$, $m\ge3$, we have $R_{10,m}=s(\G_{10,m}^\bot)=4$,
  for  $\G_{9,m}=[9,6,3]_{3^m}R_{9,m}$, $m\ge2$, it holds $R_{9,m}=s(\G_{9,m}^\bot)=3$, for  $\G_{8,m}=[8,6,2]_{3^m}R_{8,m}$, $m\ge2$, we have $R_{8,m}=s(\G_{8,m}^\bot)=2$. So, by Theorem \ref{th1:UPWS nes-suf}, these families consist of UPWS codes. All these codes are NMDS, see Propositions \ref{prop42:codesG1G1botproperty}(i),  \ref{prop42:Golay}(i). Also, we can use directly Theorem \ref{th3:NMDS-Adp1=0-UPWS}. Finally, by \cite{BasZaiZinUP,BasZin,BorgRiZin2019PIT}, $\G_{n,1}$, $n=10,9,8$, are UPWS codes.

\item[(ii)] For $\G_{8,1},\G_{9,1}$, see Proposition \ref{prop42:codesG1G1botproperty}(iv).
By \eqref{eq42:s8m bot ge2}, the codes $\G_{8,m}^\bot,~m\ge2$ are two-weight. So, by Proposition \ref{prop23:examplesCR}(ii), their dual $\G_{8,m}$ are CR with covering radius~2. Finally, for $\G_{10}^\bot=[10,4,6]_35$, see  Proposition \ref{prop42:codesG1G1botproperty}(v).

  \item[(iii)]
  Let $n=10,9,8$. By Propositions \ref{prop42:codesG1G1botproperty}(i)(iii) and \ref{prop42:Golay}(ii)(iv) with  \eqref{eq42:Ssnm}, for $\G_{n,1}^\bot=[n,n-6,6]_3R_{n,1}^\bot$ we have $R_{n,1}^\bot=5<s(\G_{n,1})=7$; for
  $\G_{n,m}^\bot=[n,n-6,6]_3R_{n,m}^\bot$ we have $R_{n,m}^\bot=6<s(\G_{n,m})=7$, $m\ge2$. Finally, by
  Proposition \ref{prop42:Golay}(ii) and \eqref{eq42:Ssnm bot},  for $\G_{10,2}=[10,6,4]_{3^2}R_{10,2}$ we have $R_{10,2}=3<s(\G_{10,2}^\bot)=4$.
   So, by Theorem \ref{th1:UPWS nes-suf}, these codes are not UPWS.

   \item[(iv)] The assertions follow from Theorem \ref{th25:NMDS}(iv) and \eqref{eq42:SmallWeights},  \eqref{eq42:small weights2-m2}, \eqref{eq23:A2ep2UP}, \eqref{eq23:A2ep1J}.

  \item[(v)]  The assertion follows from Propositions \ref{prop42:codesG1G1botproperty}(ii) and \ref{prop42:Golay}(iii) with \eqref{eq42:SmallWeights} and \eqref{eq42:small weights m2}.
   \qedhere
    \end{description}
 \end{proof}


\section{The codes related to the ternary Golay code and the corresponding lifted ones as multiple coverings}\label{sec5:MultCover}

Multiple coverings are given by Definition \ref{def1:multcov}. Let $B_j(\ov,C)$ be as in Section \ref{sec1:Intro} and $C(i)$ be as in \eqref{eq22:C(i)}. The covering quality of an $[n,k,d]_{q}R$ $(R,\lambda)$-MCF code $C$ is characterized by its \emph{MCF-$\lambda$-density} $\gamma _{\lambda }(C,R,q)\ge1$ which is the average value of $B_R(\ov,C)$ divided by $\lambda$ where the average is calculated over all  $\ov\in C(R)$. For APMCF and PMCF codes we have $\gamma _{\lambda}(C,R,q)=1$.
From the covering problem point of view, the best codes are those with small MCF-$\lambda$-density.

\begin{lemma}\label{lem52:MCF&WDcosets}
Let $C$ be an $[n,k,d]_{q}R$ code. For a weight $R$ coset $\V^{(R)}(C)$ of $C$, let $\A_R$ be the number of vectors of Hamming weight $R$ belonging to this coset; in other words, $\A_R$ is the number of the coset leaders. Then the code $C$ is $(R,\lambda)$-MCF
 if for all weight $R$ cosets, forming $C(R)$ \eqref{eq22:C(i)}, we have
 $\A_R\ge\lambda$.  If for all such cosets $\A_R=\lambda$ then $C$ is an $(R,\lambda)$-APMCF code;
 moreover, if in addition $d\ge2R$ then $C$ is an  $(R,\lambda)$-PMCF code.
\end{lemma}

\begin{proof}
  The assertions follow from Definition \ref{def1:multcov} and Theorem \ref{th22:cosets}(ii)(iii).
\end{proof}

\begin{theorem}\label{th52:RUWS=APMCF}
  Let $C$ be an $[n,k,d]_{q}R$ UPWS code. Then all its weight $R$ cosets have the same weight distribution and $C$ is an $(R,\A_R)$-APMCF code, where $\A_R$ is the number of vectors of Hamming weight $R$ belonging to any  weight $R$ coset of $C$; moreover, if in addition $d\ge2R$ then $C$ is an $(R,\A_R)$-PMCF code.
\end{theorem}

\begin{proof}
  By Theorem \ref{th1:UPWS nes-suf}, for the code $C$, we have $R=s(C^\bot)$. Now the assertions follow from Theorem~\ref{th22:cosets}(ii), citing \cite[Theorem 7.5.2]{HufPless}, and Lemma \ref{lem52:MCF&WDcosets}.
\end{proof}

\begin{example}\label{ex52}
For the codes, considered in Section \ref{sec4:newUPWS}, Tables \ref{tab5:MultCov1} and \ref{tab5:MultCov2} show the results of computer search performed using Magma \cite{Magma}, connected with multiple coverings. The calculations for $\G_{12,3}$ and $\G_{11,3}$ are not finished but they are interesting because they show that these codes are not APMCF.

We use the following notations: $s^\bot=s(C^\bot)$ is as in \eqref{eq1:support non-zero weights}; $b_R$ is the number distinct types of weight distributions of weight $R$ cosets; $\A_R$ and $\A_{R+1}$ are the number of vectors of Hamming weight $R$ and $R+1$ in weight $R$ cosets of the given type; $N$ is the number of weight $R$ cosets of the given type; $\lambda$ is the multiplicity of covering; $\gamma_\lambda$ is the MCF-$\lambda$-density;  $(R,\lambda )$-MCF, $(R,\lambda )$-PMCF, and $(R,\lambda )$-APMCF are as in Definition \ref{def1:multcov}.  In the column ``U'', $\bullet$
notes that the code $C$  is UPWS, whereas $\blacktriangledown$ says that it  is not UPWS.

\begin{table}[h]
\caption{Codes $\G_{12,m}$, $\G_{11,m}$, $\G_{11,m}^\bot$, $m\ge1$, as multiple coverings}
\centering
  \begin{tabular}{c|c|c|c|r|r|r|c|c|l}
    \hline
    code $C$&$s^\bot$&U&$b_R$&$\A_R$&$\A_{R+1}$&$N~~$&$\lambda$&$\gamma_\lambda$&$(R,\lambda)$-...MCF \\\hline

    $\G_{12,1}=[12,6,6]_33$&3&$\bullet$&1&4&9&440&4&1&(3,4)-PMCF \\\hline

     $\G_{12,2}=[12,6,6]_94$&6&$\blacktriangledown$&4&3&48&63360&3&1.28&(4,3)-MCF \\
     &&&    &3&54&190080&&\\
     &&&    &5&44&142560&&\\
     &&&    &6&48&23760&&\\\hline

     $\G_{12,3}=[12,6,6]_{27}5$&6&$\blacktriangledown$&$\ge3$&18&747&&$\le18$&$>1$&$(5,\le18)$-MCF\\
     &&&    &24&732&&&\\
     &&&    &36&702&&&&not finished\\\hline\hline

    $\G_{11,1}=[11,6,5]_32$&2&$\bullet$&1&1&6&220&1&1&(2,1)-PMCF \\\hline

    $\G_{11,2}=[11,6,5]_94$&5&$\blacktriangledown$&1&30&240&7920&30&1&(4,30)-APMCF \\\hline

    $\G_{11,3}=[11,6,5]_{27}4$&5&$\blacktriangledown$&$\ge4$&5&388&2160&$\le5$&$>1$&$(4,\le5)$-MCF\\
    &&&    &10&380&6480&&\\
    &&&    &15&372&3900&&\\
     &&&    &30&348& 291&&&not finished\\\hline\hline

        $\G_{11,1}^\bot=[11,5,6]_35$&5&$\bullet$&1&66&0&2&66&1&(5,66)-APMCF \\\hline

    $\G_{11,2}^\bot=[11,5,6]_95$&7&$\blacktriangledown$&5&21&270&5280&21&1.48&(5,21)-MCF\\
                              &&&    &30&216&528&&\\
                              &&&    &30&243&10560&&\\
                              &&&    &33&234&31680&&\\
                              &&&    &66&0&8&&\\\hline
        \end{tabular}
  \label{tab5:MultCov1}
\end{table}

\begin{table}[h]
\caption{Codes $\G_{n,m}$, $\G_{n,m}^\bot$, $n=10,9,8$, $m\ge1$,  as multiple coverings}
\centering
  \begin{tabular}{c|c|c|c|r|r|r|c|c|l}
    \hline
    code $C$&$s^\bot$&U&$b_R$&$\A_R$&$\A_{R+1}$&$N~~$&$\lambda$&$\gamma_\lambda$&$(R,\lambda)$-...MCF \\\hline

    $\G_{10,1}=[10,6,4]_32$&2&$\bullet$&1&3&12&60&3&1&(2,3)-PMCF \\\hline

    $\G_{10,2}=[10,6,4]_93$&4&$\blacktriangledown$&2&8&142&2160&8&1.24&(3,8)-MCF \\
                                 &&&&12&123&1920&&\\\hline

    $\G_{10,3}=[10,6,4]_{27}4$&4&$\bullet$&1&180&5724&112320&180&1&(4,180)-APMCF \\\hline\hline

    $\G_{10,1}^\bot=[10,4,6]_35$&7&$\blacktriangledown$&1&36&0&4&36&1&(5,36)-APMCF \\\hline

    $\G_{10,2}^\bot=[10,4,6]_96$&7&$\blacktriangledown$&1&180&360&48&180&1&(6,180)-APMCF \\\hline\hline

    $\G_{9,1}=[9,6,3]_32$&2&$\bullet$&1&9&18&8&9&1&(2,9)-APMCF \\\hline

    $\G_{9,2}=[9,6,3]_93$&3&$\bullet$&1&72&702&48&72&1&(3,72)-APMCF \\\hline

    $\G_{9,3}=[9,6,3]_{27}3$&3&$\bullet$&1&72&2970&11856&72&1&(3,72)-APMCF \\\hline\hline

    $\G_{9,1}^\bot=[9,3,6]_35$&7&$\blacktriangledown$&1&18&0&6&18&1&(5,18)-APMCF \\\hline

    $\G_{9,2}^\bot=[9,3,6]_96$&7&$\blacktriangledown$&1&72&108&144&72&1&(6,72)-APMCF \\\hline\hline

    $\G_{8,1}=[8,6,2]_31$&1&$\bullet$&1&2&13&8&2&1&(1,2)-PMCF \\\hline

    $\G_{8,2}=[8,6,2]_92$&2&$\bullet$&1&24&360&48&24&1&(2,24)-APMCF \\\hline

    $\G_{8,3}=[8,6,2]_{27}2$&2&$\bullet$&1&24&1368&624&24&1&(2,24)-APMCF \\\hline\hline

    $\G_{8,1}^\bot=[8,2,6]_{3}5$&7&$\blacktriangledown$&1&8&0&16&8&1&(5,8)-APMCF \\\hline

    $\G_{8,2}^\bot=[8,2,6]_{9}6$&7&$\blacktriangledown$&1&24&24&1344&24&1&(6,24)-APMCF \\\hline
    \end{tabular}
  \label{tab5:MultCov2}
\end{table}
\end{example}

\begin{theorem}\label{th52:Dj}
   \begin{description}
    \item[(i)] All the UPWS NMDS codes $\G_{n,m}$ of the five infinite families $\Ds_1$, $\Ds_2$, $\Ds_3$, $\Ds_4$, $\Ds_5$ are also APMCF codes; moreover some $\G_{n,1}$ codes are PCMF. We have\\
   $\Ds_1$: $\G_{12,1}=[12,6,6]_{3}3$ is $(3,4)$-PMCF,\\
     \phantom{$\Ds_1$: }$\G_{12,m}=[12,6,6]_{3^m}6$ is $(6,\lambda_{12,m})$-APMCF, $m\ge6$;\\
  $\Ds_2$: $\G_{11,1}=[11,6,5]_{3}2$  is $(2,1)$-PMCF,\\ \phantom{$\Ds_2$: }$\G_{11,m}=[11,6,5]_{3^m}5$ is $(5,\lambda_{11,m})$-APMCF, $m\ge4$;\\
    $\Ds_3$: $\G_{10,1}=[10,6,4]_{3}2$ is $(2,3)$-PMCF,\\
    \phantom{$\Ds_3$: }$\G_{10,m}=[10,6,4]_{3^m}4$ is $(4,\lambda_{10,m})$-APMCF, $m\ge3$, $\lambda_{10,3}=180$;\\
  $\Ds_4$: $\G_{9,1}=[9,6,3]_{3}2$ is $(2,9)$-APMCF,\\
   \phantom{$\Ds_4$: }$\G_{9,m}=[9,6,3]_{3^m}3$ is $(3,\lambda_{9,m})$-APMCF, $m\ge2$, $\lambda_{9,2}=\lambda_{9,3}=72$;\\
  $\Ds_5$: $\G_{8,1}=[8, 6, 2]_{3}1$ is $(1,2)$-PMCF,\\
    \phantom{$\Ds_5$: }$\G_{8,m}=[8,6,2]_{3^m}2$ is $(2,\lambda_{8,m})$-APMCF, $m\ge2$, $\lambda_{8,2}=\lambda_{8,3}=24$.

    \item[(ii)]  There are the following $(R,\A_R)$-APMCF codes which are not UPWS:\\
     $\G_{11,2}=[11,6,5]_{9}4$  is $(4,30)$-APMCF;\\
      $\G_{10,1}^\bot=[10,4,6]_35$ is $(5,36)$-APMCF, $\G_{10,2}^\bot=[10,4,6]_96$ is $(6,180)$-APMCF; \\
      $\G_{9,1}^\bot=[9,3,6]_35$ is $(5,18)$-APMCF, $\G_{9,2}^\bot=[9,3,6]_96$ is $(6,72)$-APMCF;\\
     $\G_{8,1}^\bot=[8,2,6]_{3}5$ is $(5,8)$-APMCF, $\G_{8,2}^\bot=[8,2,6]_{9}6$ is $(6,24)$-APMCF.
\end{description}
 \end{theorem}

\begin{proof}
  The assertions follow from Theorem \ref{th52:RUWS=APMCF} and Tables \ref{tab5:MultCov1}, \ref{tab5:MultCov2} .
\end{proof}

\begin{remark}\label{rem52}
\begin{description}
  \item[(i)]
By Theorem \ref{th52:RUWS=APMCF}, if a code is UPWS, then $b_R=1$. However, Tables \ref{tab5:MultCov1}, \ref{tab5:MultCov2} show that for many codes $b_R=1$ even if the code  is not UPWS.

  \item[(ii)]  By Table \ref{tab5:MultCov2} and Theorem \ref{th52:Dj}, we have  $\lambda_{9,2}=\lambda_{9,3}=72$,
  $\lambda_{8,2}=\lambda_{8,3}=24$, and $\lambda_{10,3}=180$. This gives rise to conjecture that for $m\ge2$, $\Ds_4$ and $\Ds_5$ are infinite families of $(3,72)$-APMCF and $(2,24)$-APMCF codes, respectively, and similarly, for $m\ge3$, $\Ds_3$ is infinite family of $(4,180)$-APMCF codes.

  \item[(iii)]  By Table \ref{tab5:MultCov2} and Theorem \ref{th52:Dj}, we have\\
    $\G_{9,2}=[9,6,3]_93$ is $(3,72)$-APMCF, $\G_{9,2}^\bot=[9,3,6]_96$ is $(6,72)$-APMCF;\\
     $\G_{8,2}=[8,6,2]_{9}2$ is $(2,24)$-APMCF, $\G_{8,2}^\bot=[8,2,6]_{9}6$ is $(6,24)$-APMCF;\\
       $\G_{10,3}=[10,6,4]_{27}4$ is $(4,180)$-APMCF; $\G_{10,2}^\bot=[10,4,6]_96$ is $(6,180)$-APMCF.\\
       This gives rise to the conjecture that for codes  $\G_{n,m}$ and $\G_{n,m}^\bot$ the values of $\lambda_{n,m}$ and $\lambda_{n,m}^\bot$ are equal to each other; similarly to the number of codewords of minimal weight for an NMDS code and its dual, see Theorem \ref{th25:NMDS}(iv).
\end{description}
\end{remark}

\section{Conclusions and open problems}\label{sec6:conclusion}
\subsection{New infinite families of near-MDS (NMDS) uniformly \\packed in the wide sense (UPWS) codes}
\begin{itemize}
  \item As far as it is known to the authors, the infinite families of the $m$-lifted codes $\Ds_1$, $\Ds_2$, $\Ds_3$, $\Ds_4$, $\Ds_5$ from Theorems~\ref{th41:UPWSliftGolay}(i) and \ref{th42:UPWSliftGolay}(i), apart for their 1-st members with $m=1$, have not been described in the literature as NMDS codes uniformly packed in the wide sense.
 Similarly, the codes from Theorem \ref{th42:UPWSliftGolay}(ii) have not been noted in the literature as complete regular (CR) or 2-regular. Also, it is interesting that the code  of \cite[p.\ 24, (S.13)]{BorgRiZin2019PIT}
 coincides with dual ternary Golay code, see Theorem~\ref{th41:UPWSliftGolay}(iii).

\noindent\textbf{Open problem.} For the codes $\G_{12,m}=[12,6,6]_{3^m}R_{12,m}$, $m=4,5$, find the values $R_{12,m}$ and understand if these codes are UPWS.

  \item Distinct combinatorial properties of codes from families $\Ds_j$  and their dual ones, which are described and used in Sections \ref{sec3:useful} and \ref{sec4:newUPWS}, not only helped us to prove needed lemmas, propositions, and theorems, but are also of independent interest. In particular, the following can be noted: the results on the number of codewords of
      weight $d$ and $d+1$, see Theorems~\ref{th41:UPWSliftGolay}(iv)(v) and \ref{th42:UPWSliftGolay}(iv)(v); the weight distribution, the support $S(C)$ of the set of non-zero weights, and its cardinality $s(C)$ for distinct codes $C$, see \eqref{eq1:support non-zero weights},  Propositions \ref{prop41:Golay}(iii), \ref{prop42:codesG1G1botproperty}(iii),
      \ref{prop42:Golay}(iv), and the corresponding Proofs; covering radii for distinct codes, see Propositions \ref{prop41:Golay}(ii), \ref{prop42:codesG1G1botproperty}(i), \ref{prop42:Golay}(ii).

  \item We essentially used the assertion of \cite[Theorem 1]{BasZin} that a code $C$ of covering radius $R(C)$ is UPWS if and only if $R(C)=s(C^\bot)$, see \eqref{eq1:support non-zero weights} and Theorem \ref{th1:UPWS nes-suf}. To check if the last equality holds for specific codes and code families, we interpret codes, their dual and $m$-lifted ones as NMDS codes. As far as it is known to the authors, the approach, connected with NMDS codes, is new for the codes considered in this paper. As a result, we obtained many new useful relations and assertions, see e.g. the formula of the weight distribution of an $m$-lifted NMDS code \eqref{eq3:weight distrib NMDS}, Theorem \ref{th3:NMDS-Adp1=0-UPWS} (general theorem on infinite families of UPWS NMDS codes) and so on. The properties of NMDS codes, presented in Theorem \ref{th25:NMDS}, are used in the proofs of almost all lemmas, propositions, and theorems in Sections \ref{sec3:useful} and~\ref{sec4:newUPWS}.

  \item The natural generalization of the lower bound \cite{DavTombWordsMinWeght} on the number of codewords of minimal weight $d$ in Lemma \ref{lem3:LowBndAd} increases the list of codes achieving a bound that is useful for this work.
\end{itemize}

\subsection{New almost perfect multiple covering of the farthest-off points ($(R,\lambda)$-APMCF codes)}
\begin{itemize}
  \item As far as it is known to the authors, the infinite families of the NMDS UPWS codes $\Ds_1$, $\Ds_2$, $\Ds_3$, $\Ds_4$, $\Ds_5$, have not been described in the literature as almost perfect multiple covering of the farthest-off points ($(R,\lambda)$-APMCF codes), cf. Theorem \ref{th52:Dj}(i). The sporadic examples in Theorem \ref{th52:Dj}(ii) also seem new.

\noindent\textbf{Open problem.} By a theoretic way, find values of $\lambda_{n,m}$ for the codes of infinite families $\Ds_j$ considered as multiple coverings, see Theorem \ref{th52:Dj}(i).

  \item  Remark \ref{rem52} is both interesting and useful; conjectures it contains appear reasonable though somewhat unexpected.

\noindent\textbf{Open problem.} Prove the conjectures stated in Remark \ref{rem52}(ii)(iii).
\end{itemize}

\end{document}